\documentclass[12pt]{article}
\usepackage{amsmath,amssymb,amsthm,amscd,graphicx}
\usepackage{latexsym}
\usepackage{enumerate}
\usepackage{amsbsy, amsfonts, textcomp}
\usepackage{comment}
\usepackage{array}
\usepackage[all]{xy}
\usepackage[small,nohug,heads=vee]{diagrams}
\usepackage[flushmargin]{footmisc}

\newtheorem{thm}{Theorem}[section]

\newtheorem{rem}[thm]{Remark}

\newtheorem{lema}[thm]{Lemma}
\newtheorem{coro}[thm]{Corollary}
\newtheorem{prop}[thm]{Proposition}

\newcommand{\Q}{\mathbf Q}

\newcommand{\Z}{\mathbf Z}

\newcommand{\F}{\mathbf F}

\newcommand{\Gal}{\mathrm{Gal}}

\newcommand{\Norm}{\mathrm{Norm}}

\newcommand{\GL}{\mathrm{GL}}

\newcommand{\End}{\operatorname{End}}

\newcommand{\Perm}{\operatorname{Perm}}

\newcommand{\wk} {{\widetilde{K}}}

\title{On the Galois correspondence theorem in separable Hopf Galois theory}
\author{Teresa Crespo, Anna Rio, Montserrat Vela}
\date{\today}

\begin{document}
\maketitle

\let\thefootnote\relax\footnote{T. Crespo acknowledges support by grants MTM2012-33830, Spanish
Science Ministry, and 2009SGR 1370; A.Rio and M. Vela acknowledge
support by grants MTM2012-34611, Spanish Science Ministry, and
2009SGR 1220. }

\begin{abstract}

In this paper we present a reformulation of the Galois
correspondence theorem of Hopf Galois theory in terms of groups
carrying farther the description of Greither and Pareigis. We
prove that the class of Hopf Galois extensions for which the
Galois correspondence is bijective is larger than the class of
almost classically Galois extensions but not equal to the whole
class. We show as well that the image of the Galois correspondence
does not determine the Hopf Galois structure.
\end{abstract}

\section{Introduction}

A finite field extension $K/k$ is a Hopf Galois extension if there
exists a finite cocommutative $k-$Hopf algebra   $H$  such that
$K$ is an $H-$module algebra and the $k-$linear map $j:K\otimes_k
H\to\End_k(K)$, defined by $j(s\otimes h)(t)=s(ht)$ for $h\in H$,
$s,t\in K$, is bijective. Clearly a finite Galois extension $K/k$ with Galois group $G$ is a Hopf Galois extension with Hopf
algebra the group algebra $k[G]$.

The concept of Hopf Galois extension was introduced by Chase and Sweedler to study inseparable extensions.
For a Hopf Galois extension $K/k$ with Hopf algebra $H$, they prove that the map from the set of  sub-Hopf algebras of $H$ to the
set of intermediate fields of $K/k$ sending a sub-Hopf algebra $H'$ of $H$ to the
subfield of $K$ fixed by $H'$ is inclusion reversing and injective (\cite{CS} Theorem 7.6).

Greither and Pareigis (\cite{GP}, Theorem 2.1) give a
characterization and classification of Hopf Galois structures on
separable field extensions, achieved by transforming the problem
into a group-theoretic one involving the Galois group $G$ of the
Galois closure of the field extension considered. They introduce
the subclass of almost classically Galois extensions, which can be
given a Hopf Galois structure such that the Galois correspondence
is bijective, prove that all extensions of degree smaller than
five are almost classically Galois and provide an example of a
Hopf Galois extension which is not almost classically Galois,
namely a degree 16 extension of a quadratic number field (see
\cite{GP} \S 4). In \cite{CRV}, we checked that all Hopf Galois
extensions of degree up to 7 are almost classically Galois and we
presented an example of a degree 8 extension of $\Q$ which is Hopf
Galois but not almost classically Galois (loc. cit. Example 2.1).

In \cite{Childs2000}, Childs  uses Hopf Galois structures to obtain arithmetic properties of wildly ramified extensions. A more detailed account of concepts and achievements in Hopf Galois theory can be found in \cite{CRV2}.

In contrast to what happens in the Galois case, the Hopf Galois
structure is not unique in general. This fact arises the question
on the number of different Hopf Galois structures which can be
given to a Galois extension (see e.g. \cite{Byott96},
\cite{Byott2004}). A different question which has  hardly been
considered concerns the image of the Galois correspondence for
each of the Hopf Galois structures, i.e. which intermediate fields
are fixed fields of a sub-Hopf algebra of the Hopf algebra. In this
context we may ask whether the class of Hopf Galois extensions
admitting a Hopf Galois structure for which the Galois
correspondence is bijective is bigger than the class of almost
classically extensions, if different Hopf Galois structures may
have Galois correspondences with the same image or more generally
which are the sublattices of the lattice of intermediate fields of
the extension corresponding to some Hopf Galois structure.

In this paper we give a reformulation of the Galois correspondence
theorem in term of groups, carrying farther the description of
Greither and Pareigis. We approach the question on the image of
the Galois correspondence by studying two different families of
field extensions. For the first one, we consider extensions whose
Galois closure has Galois group the Frobenius group $F_{p(p-1)}$,
for $p$ prime. We obtain non almost classically Galois extensions
which may be given a Hopf Galois structure for which the Galois
correspondence is bijective and different Hopf Galois structures
giving the same image for the Galois correspondence. For the
second family, we consider Galois extensions with Galois group the
dihedral group $D_{2p}$. In this case, each of the Hopf Galois
structures provides a different sublattice of the lattice of
intermediate extensions. The fact that there are non almost
classically Galois extensions which may be endowed with a Hopf
Galois structure giving a bijective Galois correspondence raises
the question whether all Hopf Galois extensions may be endowed
with such a Hopf Galois structure. In the last section we answer
this question negatively by exhibiting a Hopf Galois extension for
which no Hopf Galois structure gives a bijective Galois
correspondence.

\section{A reformulation of the Galois correspondence theorem in terms of groups}

In the sequel, we shall use the following notation. $K/k$ denotes
a separable extension, $n$ its degree, $\widetilde K/k$ its Galois
closure, $G=\Gal(\tilde K/k)$ and $G'=\Gal(\tilde K/K)$.

In the separable case the Hopf algebras giving a Hopf Galois
structure are forms of some group algebra. More precisely
they are Hopf algebras of the form $\wk[N]^G$ with $N$ as in \cite{GP} Theorem 2.1. The
next two propositions give a description of the sub-Hopf algebras
of a group algebra and of the algebra $\wk[N]^G$.

\begin{prop}
Let $k$ be a field and $G$ a finite group. The sub-Hopf algebras of  $k[G]$ are the group algebras
 $k[H]$, with $H$ a subgroup of  $G$.
\end{prop}

\noindent {\it Proof.} Let $G=\{ g_1,g_2,\dots,g_n \}$ and $C$ a
sub-Hopf algebra of $k[G]$. In particular, $C$ is a vector
subspace of  the $k$-vector space $k[G]$, for which
$g_1,g_2,\dots,g_n$ is a basis. Gauss reduction and reordering of
the elements of $G$, if necessary, provides a basis
$(v_1,\dots,v_m)$ of $C$, where

$$v_j  =  g_j+\sum_{i>m} \lambda_i^j g_i, \quad 1\leq j \leq m, $$

\noindent with $\lambda_i^j \in k$. Since $C$ is sub-coalgebra, we should
have $\Delta(v_j) \in C \otimes C, $ for all $j=1,\dots m$, where
$\Delta$ denotes the  coproduct in $k[G]$. We have

$$\Delta(v_j)  =  g_j\otimes g_j +\sum_{i>m} \lambda_i^j g_i\otimes g_i.$$

\noindent $C\otimes C$ has basis $\{ v_i \otimes v_j\}_{1\leq i,j\leq m}$
and expressing every  $v_i \otimes v_j$ in the basis $\{g_i\otimes
g_j\}_{1\leq i,j\leq n}$ of $k[G]\otimes k[G]$ we obtain {\small
$$\begin{array}{lll} v_1\otimes v_1 & = & g_1\otimes g_1 +\sum_{i>m} \lambda_i^1 g_1\otimes g_i +\sum_{i>m} \lambda_i^1 g_i\otimes g_1+\sum_{i>m,j>m} \lambda_i^1 \lambda_j^1 g_i\otimes g_j  \\v_1\otimes v_2 & = & g_1\otimes g_2 +\sum_{i>m} \lambda_i^2 g_1\otimes g_i +\sum_{i>m} \lambda_i^1 g_i\otimes g_2+\sum_{i>m,j>m} \lambda_i^1 \lambda_j^2 g_i\otimes g_j \\ && \dots \end{array} $$
}

When we write  $\Delta(v_1)$ in the basis of $C\otimes C$, we see
that $g_1\otimes g_1$ only appears in the expression of
$v_1\otimes v_1$. In general, in the expression of $v_i\otimes
v_j$ there appears $g_i\otimes g_j$ and this term does not appear
in any other $v_{i'}\otimes v_{j'}$. Therefore, the only
 possibility is $\Delta(v_1)=v_1\otimes v_1$, which implies $\lambda_i^1=0, i=m+1,\dots,n$.
Analogously, we obtain $\lambda_i^j=0, i=m+1,\dots,n$, for $j=2,\dots,m$.

Finally, since $C$ has basis $(g_1,g_2,\dots,g_m)$, as a
$k$-vector subspace of $k[G]$ and  $C$ is a subalgebra,
$\{g_1,g_2,\dots,g_m\}$ is a subgroup of $G$. \hfill $\Box$

\vspace{0.5cm}

For $G$ and $G'$ as above, we denote by $\lambda$ the morphism from $G$ into the symmetric group $S_n$ given by the action
of $G$ on the left cosets $G/G'$ by left translation.

\begin{prop} Let $N$ be a regular subgroup of $S_n$ normalized by $\lambda(G)$. \\
There is a bijection between the set of $k$-sub-Hopf algebras of
$\wk[N]^G$ and the set of subgroups of $N$ stable under the action
of $\lambda(G)$ by conjugation.

\end{prop}

\noindent {\it Proof.} We consider the map

$$\begin{array}{cccc}
\Phi: & \{ \textrm{subgroups of } N  \} & \rightarrow  &\{
k\textrm{-sub-Hopf algebras of } \wk[N]^G \} \\ & N' & \mapsto &
(\wk[N'])^G
\end{array}$$

\noindent where $(\wk[N'])^G=\{  x \in \wk[N'] | \sigma(x)=x,
\forall \sigma \in G \}$ and  for $x=\sum_{\tau \in N'} a_{\tau}
\tau \in \wk[N]$ and $\sigma \in G$, $\sigma(x)=\sum_{\tau \in N'}\sigma(a_{\tau})
\lambda(\sigma)\tau \lambda(\sigma)^{-1}$.

We want to prove that, if $N'$ is a subgroup of $N$, then
$\Phi(N')=\Phi(\overline{N'})$, where $\overline{N'}=\cap_{\sigma
\in G} \lambda(\sigma) N' \lambda(\sigma)^{-1}$. Clearly
$\overline{N'}\subset N' \Rightarrow \Phi(\overline{N'})\subset
\Phi(N')$. If $x=\sum_{\tau \in N'} a_{\tau} \tau \in
\Phi(N')=\wk[N]^G$, then $\sigma(x)=x$, for all $\sigma \in G$ and
so

\begin{equation}\label{eq}
\sum_{\tau \in N'} a_{\tau} \tau=\sum_{\tau \in N'}
\sigma(a_{\tau}) \lambda(\sigma) \tau \lambda(\sigma)^{-1}, \,
\forall \sigma \in G.
\end{equation}

\noindent If $\rho \in N' \setminus \overline{N'}$, there exists
$\sigma \in G$ such that $\rho \not \in \lambda(\sigma) N'
\lambda(\sigma)^{-1}$. Hence equality (\ref{eq}) for this
$\sigma$, implies $a_{\rho}=0$. We have then $(\wk[N'])^G \subset
\wk[\overline{N'}]$ and so, $(\wk[N'])^G \subset
(\wk[\overline{N'}])^G$.

Let us see now that $\Phi$ is surjective. Let $H$ be a
$k$-sub-Hopf algebra of $\wk[N]^G$. Then $H\otimes \wk$ is a
$\wk$-sub-Hopf algebra of $\wk[N]^G \otimes \wk \simeq \wk[N]$ and
so $H\otimes \wk \simeq \wk[N']$ for some subgroup $N'$ of $N$.
Since $H$ is a $k$-algebra, we have $H=(H \otimes
\wk)^G=\wk[N']^G=\Phi(N')$. We have then a surjective map

$$\begin{array}{ccc}
 \{ \textrm{stable subgroups of } N  \} & \rightarrow  &\{
k\textrm{-sub-Hopf algebras of } \wk[N]^G \} \\  N' & \mapsto &
(\wk[N'])^G
\end{array}$$

\noindent Indeed, for any $k$-sub-Hopf algebra $H$ of $\wk[N]^G$,
we have $H=\Phi(N')$ for some subgroup $N'$ of $N$ and
$\Phi(N')=\Phi(\overline{N'})$, with $\overline{N'}$ stable. It
remains to prove that different stable subgroups give different
subalgebras. But, if $N'$ is a subgroup of $N$, stable under the
action of $\lambda(G)$, the elements in $\Phi(N')$ are precisely
the elements of the form $\sum_C a_C (\sum_{\tau \in C} \tau)$,
where $C$ runs over the conjugation classes of $S_n$ having
nonempty intersection with $N'$. \hfill $\Box$

\vspace{0.5cm}

Now, the Galois correspondence theorem can be reformulated in the following way:

\begin{thm} If $K/k$ is a Hopf Galois extension with Hopf algebra
$H=\wk[N]^G$
for a regular subgroup $N$ of $Perm(G/G')$, then the map \\
\empty \qquad \qquad\\
$
\begin{array}{rcl}
{\cal F}_N:\{\mbox{Subgroups }N' \subseteq N  \textrm{ stable under } \lambda(G)\} &\longrightarrow&\{\mbox{Fields }E\mid k\subseteq E\subseteq K\}\\
N'&\mapsto &K^{\wk[N']^G}
\end{array}
$ \\

\noindent is injective and inclusion reversing.
\end{thm}

\begin{rem}\label{rem} {\rm If the regular subgroup $N$ of $S_n$ normalized by
$\lambda(G)$ is not contained in the alternating group $A_n$, then
the subgroup $N_1:=N\cap A_n$ is stable under conjugation by
$\lambda(G)$ and has index 2 in $N$. We have $\dim_k
\widetilde{K}[N]^G=n$ and $\dim_k \widetilde{K}[N_1]^G=n/2$. The
field ${\cal F}_N(N_1)$ is then a quadratic extension of~$k$.}
\end{rem}

\section{A family of Hopf Galois extensions}

We consider an extension $K_0/k$ of prime degree $p\geq 5$ with Galois closure $\widetilde{K}$ such that the Galois group $G$ of $\widetilde{K}|k$ is the Frobenius group $F_{p(p-1)}$. The group $\Gal(\tilde K/K_0)$ is a Frobenius complement of $F_{p(p-1)}$. Let $d$ be a divisor of $p-1$, $1<d<p-1$, $G'$  the subgroup of  $\Gal(\tilde K/K_0)$ with index $d$ and $K=\widetilde{K}^{G'}$ the subfield of $\widetilde{K}$ fixed by $G'$. We shall study the extensions $K|k$. Let us note that over $k=\Q$ the Eisenstein polynomial $X^p-p$ has Galois group $F_{p(p-1)}$, hence the extension $K_0/\Q$ obtained by adjoining a root of this polynomial satisfies the above conditions.

\subsection{Hopf Galois character}

\begin{prop}
The extension $K|k$ is Hopf Galois, for all prime $p\geq 5$.
\end{prop}

\noindent {\it Proof.} We consider the tower of fields $k\subset K_0 \subset K \subset \widetilde{K}$. The extension $K_0/k$  is Hopf Galois, since it is a prime degree extension such that its Galois closure has a solvable group (see \cite{Childs89}); the extension $K/K_0$ is Galois since $\widetilde{K}/K_0$ is cyclic. This implies that $K/k$ is Hopf Galois by \cite{CRV} theorem 6.1. \hfill $\Box$

\vspace{0.5cm}
Let us see now if $K/k$ is almost classically Galois. Let us recall that the group $F_{p(p-1)}$ can be seen as a subgroup of $\GL(2,p)$, namely

$$F_{p(p-1)} = \left\{ \left( \begin{array}{cc} 1 & b \\ 0 & c \end{array} \right) : b \in \F_p, c \in \F_p^* \right\} \subset \GL(2,p).$$

\noindent For each divisor $d$ of $p-1$, $F_{p(p-1)}$ has a unique normal subgroup $F_{pd}$ corresponding to the unique subgroup $C_{d}$ of order $d$ of $\F_p^*$:

$$F_{pd}=\left\{\begin{pmatrix} 1&b\\0&c\end{pmatrix} : b\in\F_p, c \in C_{d} \right \}.$$

\noindent Since $G'$ has order $(p-1)/d$, the only candidate to be a normal complement for $G'$ in $G$ is $F_{pd}$. This group contains all elements of order $\ell$ in $F_{p(p-1)}$, for every divisor $\ell$ of $d$, which are of the form $\left(\begin{smallmatrix} 1&b\\0&c \end{smallmatrix}\right)$, with $c$ of order $\ell$ in $\F_p^*$. We obtain then two cases, depending on $D:=\gcd((p-1)/d,d)$.

\begin{itemize}
\item If $D \neq 1$, let $\ell$ be a prime number dividing $D$. The group $G'$ has some element of order $\ell$, since $\ell$ is a prime number dividing $|G'|$. Therefore, $G'$ and $F_{pd}$ have nontrivial intersection, $G'$ has no normal complement in $G$ and $K/k$ is not almost classically Galois.
\item If $D = 1$, the subgroups $F_{pd}$ and $G'$ intersect in $0$, hence $G'$ has a normal complement, the extension $K/k$ is almost classically Galois and the structure is given by $F_{pd}$.
\end{itemize}

We state what we have proved so far in the following proposition.

\begin{prop}
Let $p \geq 5$ be a prime number, $d$ a nontrivial divisor of $p-1$. An extension $K/k$ of degree $pd$ such that its Galois closure $\widetilde{K}$ has Galois group over $k$ the Frobenius group $F_{p(p-1)}$ is almost classically Galois if and only if $\gcd((p-1)/d,d)=1$.
\end{prop}

We shall now endow the extension $K/k$ with two different Hopf Galois structures. We consider two groups of order $pd$, the cyclic group $C_{pd}$ and the Frobenius group $F_{pd}$. Let us note that for a prime $d$, these are the unique groups of order $pd$.
We fix a generator $\zeta$ of $\F_p^*$ and write

$$
S=\begin{pmatrix} 1&0\\0&\zeta\end{pmatrix} \qquad T=\begin{pmatrix} 1&1\\0&1\end{pmatrix}
$$
Then
$$
G=\big\langle S, \ T\big\rangle=\left\{ S^iT^j=\begin{pmatrix} 1&j\\0&\zeta^i\end{pmatrix}\right\}_{j\bmod p, i\bmod p-1}
$$

\noindent We take $G'=\langle S^d\rangle$ and the left transversal of $G/G'$
$$
T^xS^m=\begin{pmatrix} 1&x\zeta^m\\0&\zeta^m\end{pmatrix}, \quad x\in\F_p,  \quad  0\le m<d.
$$

\noindent Let us note that the element $\left(\begin{smallmatrix} 1&b\\0&\zeta^i\end{smallmatrix}\right)$ of $G$  is in the class determined by $m=i \bmod d, \ x=b\zeta^{-i}$.
Let us identify the set $G/G'$ with $X=(\Z/(d)) \times \F_p$ by $(m,x) \leftrightarrow T^xS^m$. By computation of the action of the generators $T,S$ of $G$ on $G/G'$, we obtain that the image $\lambda(G)$ of $G$ in $S_{pd}=\Perm(X)$ is generated by

$$
\sigma_1: (m,x)\mapsto (m, \ x+1), \quad \sigma_2: (m,x)\mapsto (m+1, \ x\zeta^{-1}).
$$

The subgroup $F$ of $G$ isomorphic to $F_{pd}$ is $F=\langle S^{\frac{p-1}d}, \ T\rangle$. It can be seen as a regular subgroup of the symmetric group $S_{pd}$ via the action on itself by left translation. We identity $F$ with the set $X=(\Z/(d))\times \F_p$ by $(m,x) \leftrightarrow S^{m\frac{p-1}d}T^x$. By computation of the action of  the generators $T$ and $S^{\frac{p-1}d}$, we obtain that the image $N_1$ of $F$ in $\Perm(X)$ is generated by

$$
\tau_1: (m,x)\mapsto (m,  x+\zeta^{m\frac{p-1}d}), \quad \tau_2: (m,x)\mapsto (m+1, x).
$$

\noindent Let us see now that $\lambda(G)$ normalizes $N_1$. By computation, we obtain

$$\begin{array}{ccc} \sigma_1\tau_1\sigma_1^{-1}=\tau_1 & , &
\sigma_1\tau_2\sigma_1^{-1}=\tau_2 \\ \sigma_2\tau_1\sigma_2^{-1}=\tau_1^{\zeta^{-1-\frac{p-1}d}} & , &  \sigma_2\tau_2\sigma_2^{-1}=\tau_2.\end{array}
$$

We consider now the cyclic group $C:=C_{pd}$ of order $pd$. It can be seen as a regular subgroup of the symmetric group $S_{pd}$ via the action on itself by left translation. We have $C=C_d \times C_p$, so we may identity it with the set $X=(\Z/(d))\times \F_p$ in the obvious way. By computing the action of the generator $(1,1)$ of $C$, we obtain that the image $N_2$ of $C$ in $\Perm(X)$ is generated by

$$ \tau: (m,x) \mapsto (m+1,x+1).$$

\noindent Let us see now that $\lambda(G)$ normalizes $N_2$. By computation, we obtain

$$\sigma_1 \tau \sigma_1^{-1}=\tau , \quad \sigma_2 \tau \sigma_2^{-1}=\tau^k,$$

\noindent where $k$ is the integer in the range $[0,pd-1]$ determined by $k\equiv 1 \pmod{d}, k \equiv \zeta^{-1} \pmod{p}$.

\vspace{0.5cm}
We have then obtained the following result.

\begin{thm} Let $p \geq 5$ be a prime number, $d$ a nontrivial divisor of $p-1$. An extension $K/k$ of degree $pd$ such that its Galois closure $\widetilde{K}$ has Galois group over $k$ the Frobenius group $F_{p(p-1)}$ has (at least) two Hopf Galois structures given by the cyclic and Frobenius groups, respectively. The extension $K/k$ is almost classically Galois if and only if $\gcd((p-1)/d,d)=1$ and in this case the structure is given by the group $F_{pd}$.

\end{thm}

\subsection{The Galois correspondence}

We may determine the intermediate fields of the extension $K/k$ by classical Galois theory applied to the Galois extension $\widetilde{K}/k$. Since $K$ is the subfield of $\widetilde{K}$ fixed by a cyclic group $G'$ of $G$, the intermediate fields of $K/k$ are in one-to-one correspondence to the subgroups of $G=F_{p(p-1)}$ containing $G'$. Writing again

$$F_{p(p-1)} = \left\langle S,T \right\rangle,$$

\noindent we shall describe the subgroups of $F_{p(p-1)}$. These are

\begin{enumerate}[$\bullet$]
\item the subgroups

$$F_{pd}= \left\langle S^{\frac{p-1} d}, T \right\rangle,$$

\noindent which are all normal subgroups and satisfy $F_{pd_1} \subset F_{pd_2}$ if and only if $d_1|d_2$. In particular, for $d=1$, we obtain the $p$-Sylow subgroup.
\item for each divisor $d$ of $p-1$ we have $p$ cyclic subgroups $C_d(b)$, $b \in \F_p$, of order $d$ which are all conjugate:

$$C_d(b)= \left\langle S^{\frac{p-1} d} T^b \right\rangle.$$

\noindent We have $C_{d_1}(b) \subset C_{d_2}(b)$ if and only if $d_1|d_2$;  $C_d(b_1)\cap C_d(b_2)=1$ if $b_1 \neq b_2$ and $C_d(b) \subset F_{pd}$, for all $b \in \F_p$.
\end{enumerate}

Since we fixed $G'=C_{(p-1)/d}(0)$, the subgroups of $G$ containing $G'$ are the groups $C_{(p-1)/d'}(0)$ and $F_{p(p-1)/d'}$, with $d'$ running over the divisors of $d$. Hence, for each divisor $d'$ of $d$, there is a field $L_{d'}^1$ such that $K_0 \subset L_{d'}^1 \subset \widetilde{K}$ and $[L_{d'}^1:K_0]=d'$ and a field $L_{d'}^2$ such that $[L_{d'}^2:k]=d'$, $L_{d'}^2 \cap K_0=k$ and $L_{d'}^2 \subset L_{d'}^1$. Moreover if $d_1|d_2$, then $L_{d_1}^1 \subset L_{d_2}^1$ and $L_{d_1}^2 \subset L_{d_2}^2$.

Let us look now at the subgroups of the two groups $N$ giving a Hopf Galois structure to $K/k$. For each $d'$ dividing $d$, the cyclic group $C_{pd}$ has exactly one subgroup of order $pd'$ and exactly one of order $d'$, namely the cyclic groups $C_{pd'}$ and $C_{d'}$, hence the Hopf Galois structure of type $C_{pd}$  yields a bijective Galois correspondence if and only if all of them are stable under the action of $G$. This is clear by the fact that $C_{pd}$ has a unique subgroup for each order. For each $d'$ dividing $d$, the group $F_{pd}$ has a subgroup $F_{pd'}$ of order $pd'$ and $p$ conjugate subgroups $C_{d'}(b)$ of order $d'$, hence the Hopf Galois structure of type $F_{pd}$  yields a bijective Galois correspondence if, for each $d'$, the subgroup $F_{pd'}$ and one of the subgroups $C_{d'}(b)$ are stable under the action of $G$. For $F_{pd'}$ this is clear. For $C_{d'}(b)$, since $\sigma_1$ normalizes $N_1$, it is enough to compute $\sigma_2 \tau_2^e \tau_1^b \sigma_2^{-1}$, where $e=d/d'$. We obtain
$$\sigma_2 \tau_2^e \tau_1^b \sigma_2^{-1}=\tau_2^e \tau_1^{b(\zeta^{-1}-\frac{p-1} d)}.$$

\noindent Now the powers of $\tau_2^e \tau_1^b$ are $(\tau_2^e \tau_1^b)^k=\tau_2^{ke} \tau_1^{b(1+\eta^e+\dots+\eta^{(k-1)e})}$, for $\eta=\zeta^{\frac{p-1} d}$, so the subgroup $C_{d'}(b)$ is stable exactly for $b=0$.

\vspace{0.5cm}
We have then obtained the following result.

\begin{thm} Let $p \geq 5$ be a prime number, $d$ a nontrivial divisor of $p-1$. Let $K/k$ be an extension of degree $pd$ such that its Galois closure $\widetilde{K}$ has Galois group over $k$ the Frobenius group $F_{p(p-1)}$. We can endow $K/k$ with a non almost classically Galois Hopf Galois structure of type $C_{pd}$ such that the Galois correspondence is one-to-one. We can also endow $K/k$ with a Hopf Galois structure of type $F_{pd}$, which is almost classically Galois exactly when $\gcd((p-1)/d,d)=1$ and such that the Galois correspondence is always one-to-one.

\end{thm}

\begin{coro} There exist Hopf Galois extensions which are not almost classically Galois but may be endowed with a Hopf Galois structure such that the Galois correspondence is one-to-one.
\end{coro}

\section{A family of Galois extensions}

Let $p\ge 3$ be a prime number and $K/k$ a Galois extension with Galois group $G=D_{2p}$, the dihedral group  of order $2p$. The Hopf Galois structures of $K/k$ are determined in \cite[Thm. 6.2]{Byott2004}. There are $2+p$ structures of which $p$ are of type $C_{2p}$ and the two others correspond to $D_{2p}$ (the classical Galois one) and to its opposite group. We shall describe these Hopf Galois structures and see that the images of the corresponding Galois correspondences are all different.

We shall work with the following presentation of $D_{2p}$.

$$\begin{array}{rl}
D_{2p}=&<\sigma, \tau \mid \sigma^p=1, \tau^2=1, \tau \sigma=\sigma^{p-1} \tau> \\
&=\{ 1, \sigma, \sigma^2, \dots, \sigma^{p-1}, \tau,  \tau \sigma, \tau \sigma^2 , \dots, \tau \sigma^{p-1} \}
\end{array}.$$

We consider the embedding of $G$ in $S_{2p}$ given by the action of $G$ on itself by left translation.

$$\begin{array}{ll}
\lambda:&D_{2p} \hookrightarrow Perm(D_{2p})\simeq S_{2p} \\
&  g \mapsto \lambda(g):x \mapsto gx.
\end{array} $$

The Galois structure corresponds to $\rho(D_{2p})$, where $\rho: D_{2p} \hookrightarrow Perm(D_{2p})$ \linebreak $\simeq S_{2p}$ is given by $\rho(g)(x)=xg^{-1}$. The image of the Galois correspondence is then the whole lattice of intermediate fields and is given by the fundamental theorem of Galois theory. Since $D_{2p}$ has two conjugation classes of nontrivial subgroups, one class containing the normal subgroup of order $p$, which is generated by $\sigma$ and another class of length $p$ of subgroups of order $2$ generated by the elements $ \tau \sigma^k$, $k=0, \dots, p-1$, the subfield lattice consists, besides $k$ and $K$, in one normal extension of degree 2, and $p$ conjugate extensions of degree $p$.

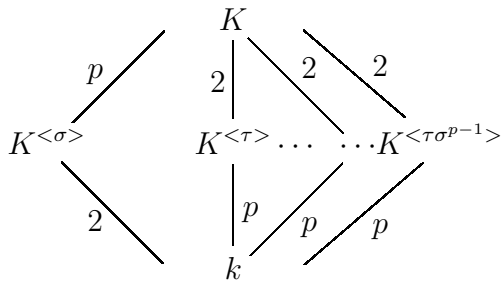
$$\begin{diagram}
& &  & K & &  &\\
& \ruLine^{p}  &  &  \uLine^{2}  & \luLine^{2} & \luLine^{2}& \\
K^{<\sigma>} &  & & K^{<\tau >}&  \cdots & \cdots &   K^{< \tau \sigma^{p-1}>} \\
& \luLine_{2}  &  &\uLine_{p} &\ruLine_{p} &  \ruLine_{p}& \\
&  & & k & & &
\end{diagram}
$$

The nonclassical structure of type $D_{2p}$ is given by $\lambda(G)$. In this case, the image of the Galois correspondence consists, besides $k$ and $K$, in the normal extension of degree 2 (see \cite{GP} Theorem 5.3).

$$\begin{array}{c}
K \\
p \mid \\
K^{<\sigma>}\\
2 \mid \\
k
\end{array}
$$

In order to study the cyclic structures, we identify $D_{2p}$ with the set $Y=\F_2 \times \F_p$ by $\tau^i \sigma^j \leftrightarrow (i,j)$ and $S_{2p}$ with $\Perm(Y)$. We have then

$$\lambda(\sigma)(m,n)=(m,n+m), \lambda(\tau)(m,n)=(m+1,n).$$

\noindent Let us now write $C_{2p}=C_2\times C_p$ and identity it with $Y$. We have then $p$ embeddings of $C_{2p}$ in $\Perm(Y)$ given by sending a chosen generator of $C_{2p}$ to the permutation

$$\pi_c:(m,n) \mapsto (m+1,n+1+(-1)^m c), \quad c=0,\dots,p-1,$$

\noindent which is a $2p$-cycle. Let us denote $N_c=\langle \pi_c \rangle.$ Each of the groups $N_c\simeq C_{2p}$ have just two proper nontrivial subgroups of orders $2$ and $p$, which are normal. Since $N_c$ is not contained in the alternating group $A_{2p}$, by Remark \ref{rem}, the subgroup $N_c \cap A_{2p}=<\pi_c^2>$ is stable under conjugation by $\lambda(G)$. Then the field $K^{K[<\pi_c^2>]^G}=K^{<\sigma>}$ is in the image of the Galois correspondence theorem for all Hopf Galois structures. The subgroup $<\pi_c^p>$ is also stable under conjugation by $\lambda(G)$, since $\lambda(\sigma) \pi_c \lambda(\sigma^{-1})= \pi_c^p$. It corresponds then to the intermediate field $K^{K[<\pi_c^p>]^G}$, which has degree $p$ over $k$. We have

$$\begin{array}{rl}
K^{K[<\pi_c>^p]^G}=&\{ x \in K \mid \mu(h)(x)=\varepsilon(h)(x), \forall h \in k[<\pi_c^p>] \} =\\
=& \{ x \in K \mid \mu(\pi_c^p)(x)=\varepsilon(\pi_c^p)(x)=x\} =\\
=& \{ x \in K \mid \tau \sigma^{c}(x)=x \}=K^{<\tau\sigma^{c} >}
\end{array} $$

\noindent since the action of $\mu$ is given by $\mu(\pi_c^p)(x)=\pi_c^{-p}(1_G)(x)= \tau \sigma^{c}(x)$, $1_G$ is identified with $(0,0) \in Y$ and $\pi_c^p(0,0)=(1,c)$ corresponds to $\tau \sigma^c$.

We have then obtained that for each of the Hopf Galois cyclic structures of $K/k$ there is exactly one extension of degree $p$ in the image of the Galois correspondence. Its image gives the lattice

$$
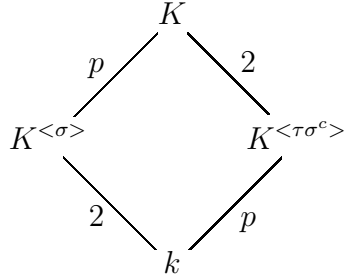
\begin{diagram}
& & K & &\\
&\ruLine^{p} & &\luLine^{2} & \\
K^{<\sigma>}& &  &  & K^{<\tau \sigma^{c}>}\\
&\luLine_{2} & &\ruLine_{p} &\\
& & k& &
\end{diagram}
$$
for $c$ an integer, $1 \leq c \leq p$.

\section{A Hopf Galois extension with non bijective Galois correspondence}

In this section we exhibit a Hopf Galois extension such that the
Galois correspondence is not bijective for any of its Hopf Galois
structures. The extension considered is a separable field
extension of degree 12 such that the Galois group of its normal
closure is isomorphic to $F_{18}:2$. As an example of such an
extension we can take $k=\Q$ and $K$ to be the field obtained by
adjoining to $\Q$ a root of the polynomial
$$
x^{12} - 2x^{11} - 2x^9 + 15x^8 - 4x^7 - 12x^6 - 4x^5 + 15x^4 -
2x^3 - 2x + 1.
$$

\begin{prop}Let $K/k$ be a separable field extension of degree 12 such that the Galois group $G=\Gal(\tilde K/k)$ of its normal closure is isomorphic to
$F_{18}:2\simeq S_3\times S_3$. Then, $K/k$ is a Hopf Galois
extension non almost classically Galois.
\end{prop}

\noindent {\it Proof.} The extension $K/k$ is not almost
classically Galois since $F_{18}:2$ has no normal subgroups of
order 12.

Since all transitive subgroups of $S_{12}$ isomorphic to
$S_3\times S_3$ lie in the same conjugacy class, there is an
enumeration of the left cosets $G/G'$ such that the embedding
$\lambda$ of $G$ in $S_{12}$ obtained via the action on those left
cosets is
$$
\lambda(G)=\langle \sigma,\ \tau \rangle$$ where
$$\begin{array}{l}
\sigma=(1,2,3,4,5,6)(7,8,9,10,11,12),\\
\tau=(1,9)(2,10)(3,7)(4,8)(5,11)(6,12).
\end{array}
$$
The group
$$
N=\langle (1, 11, 5, 9, 3, 7)(2, 12, 6, 10, 4, 8), (1,10)(2,
9)(3,8)(4,7)(5,12)(6,11)\rangle,
$$

\noindent isomorphic to the dihedral group $D_{2\cdot 6}$, is a
regular subgroup of $S_{12}$ such that $ \lambda(G)\subset
\Norm_{S_{12}}(N). $ Therefore $N$ provides a Hopf Galois
structure for $K/k$. \hfill $\Box$

\vspace{0.5cm}

We look now for a regular subgroup $N\subset S_{12}$ normalized by
$\lambda(G)$ and such that the lattice of its subgroups stable
under the action of $\lambda(G)$ by conjugation is in one-to-one
correspondence with the lattice of intermediate fields of $K/k$.
That is, we want to know whether $K/k$ may be endowed with a Hopf
Galois structure such that the Galois correspondence is bijective.

By classical Galois theory, the lattice of intermediate fields of
$K/k$ corresponds to the subgroups of $G$ containing $G'$. This
gives

\medskip

\begin{center}
\begin{tabular}{c | c}
Subgroups of $G$ containing $G'$ & Intermediate fields of $K/k$ \\
\hline
$G'$ & $K$ \\
3 subgroups  of order 6 & 3 extensions of  $k$ of degree 6\\
1 subgroups  of order 9  &1 biquadratic extension of  $k$\\
3 subgroups  of order 18  &3 quadratic extensions of $k$\\
$G$ & $k$\\
\end{tabular}
\end{center}

Checking over the 5 isomorphism classes of groups of order 12, we
see that only the dihedral group has enough subgroups to be in
bijective correspondence with such a subfield lattice. Therefore,
the question becomes:

{\it Is there any  regular subgroup $N\subset S_{12}$ isomorphic
to the dihedral group $D_{2\cdot 6}$ such that $\lambda(G)\subset
\Norm_{S_{12}}(N)$ and $N$ has 3 subgroups of order 2, 1 subgroup
of order 3,  3 subgroups of order 6, all of them stable under
conjugation by $\lambda(G)$?}

\begin{lema} The three subgroups of order 2 of $N$ stable under conjugation by $G$ must be the subgroups
$\langle \omega_i\rangle$ where
$$\begin{array}{l}
\omega_1=(1, 7)(2, 8)(3, 9)(4, 10)(5, 11)(6, 12)\\
\omega_2=(1, 9)(2, 10)(3, 11)(4, 12)(5, 7)(6, 8)\\
\omega_3=(1, 11)(2, 12)(3, 7)(4, 8)(5, 9)(6, 10).
\end{array}
$$
\end{lema}
\begin{proof}
Since $N$ is regular, its elements of order 2 have no fixed point
and then must be the product of six disjoint transpositions.
Looking for an element $\omega$ which is the product of 6 disjoint
transpositions and is stable under conjugation by $\lambda(G)$
amounts to look for a set of 6 disjoint transpositions which are
permuted by conjugation of the two generators of $\lambda(G)$. By
performing the computation, we obtain the result in the lemma.
\end{proof}

\begin{prop}
There is no regular subgroup $N\subset S_{12}$ isomorphic to the
dihedral group $D_{2\cdot 6}$ such that $\lambda(G)\subset
\Norm_{S_{12}}(N)$ and $N$ has 3 subgroups of order 2 and 3
subgroups of order 6 all of them stable under  conjugation by $\lambda(G)$.
\end{prop}
\begin{proof}

Let $N$ be a regular subgroup of $S_{12}$ isomorphic to the
dihedral group $D_{2\cdot 6}$.  The dihedral group $D_{2\cdot
6}=\langle r,s\mid  r^6=s^2=1, rs=sr^5\rangle$ has exactly 3
subgroups of order 6 and one of them is the cyclic group $\langle
r \rangle$. If $N$ has 3 subgroups of order 6 stable under
conjugation by $\lambda(G)$, then, in particular, its cyclic
subgroup of order 6 must be. This implies that the center of $N$
is stable under $\lambda(G)$. Now, if $N$ has 3 subgroups of order
2 stable under conjugation by $\lambda(G)$, according to the
lemma, these must be generated by $\omega_1,\omega_2$ and
$\omega_3$, respectively. Then one of the elements $\omega_i$
should commute with the other two, which does not hold.
\end{proof}

As a corollary, we obtain the following result:

\begin{thm}
There exist separable Hopf Galois extensions $K/k$ such that the
Galois correspondence is not bijective for any of its Hopf Galois
structures.
\end{thm}

\vspace{1cm}

\footnotesize

\noindent Teresa Crespo, Departament d'\`Algebra i Geometria,
Universitat de Barcelona, Gran Via de les Corts Catalanes 585,
E-08007 Barcelona, Spain, e-mail: teresa.crespo@ub.edu

\vspace{0.3cm} \noindent Anna Rio, Departament de Matem\`atica
Aplicada II, Universitat Polit\`ecnica de Catalunya, C/Jordi
Girona, 1-3 Edifici Omega, E-08034 Barcelona, Spain, e-mail:
ana.rio@upc.edu

\vspace{0.3cm} \noindent Montserrat Vela, Departament de
Matem\`atica Aplicada II, Universitat Polit\`ecnica de Catalunya,
C/Jordi Girona, 1-3 Edifici Omega, E-08034 Barcelona, Spain,
e-mail: \linebreak montse.vela@upc.edu


\begin{thebibliography}{1}
\bibitem{Byott96} N.P. Byott, \emph{Uniqueness of Hopf Galois structure for separable field extensions}. Comm.  Algebra, 24 (1996), 3217-3228. Corrigendum, ibid., 3705.
\bibitem{Byott2004} N. P. Byott, {\em Hopf Galois structures on Galois field extensions of degree pq}, J. Pure and Applied Algebra  {\bf 188} (2004), 45-57.
\bibitem{CS} S.U. Chase, M.E. Sweedler, \emph{Hopf algebras and Galois theory}, Springer 1969.
\bibitem{Childs89} L. Childs, \emph{On the Hopf Galois theory for separable field extensions}. Comm. Algebra 17 (1989), 809-825.
\bibitem{Childs2000} L. Childs,  \emph{Taming Wild Extensions: Hopf Algebras and Local Galois Module Theory}.
Mathematical Surveys and Monographs Series, vol. 80. American Mathematical Soc. (2000).
\bibitem{CRV} T. Crespo, A. Rio, M. Vela, \emph{The Hopf Galois property in subfield lattices}, submitted, arXiv:1309.5754.
\bibitem{CRV2} T. Crespo, A. Rio, M. Vela, \emph{From Galois to Hopf Galois: theory and practice}, submitted, arXiv:1403.6300.
\bibitem{GP} C. Greither, B. Pareigis, \emph{Hopf Galois theory for separable field extensions}. J. Algebra, 106 (1987), 239-258.
\end{thebibliography}
\end{document}